\documentclass[a4paper,11pt,reqno]{amsart}
\usepackage{amsmath}
\usepackage{amsthm,enumerate}

\usepackage{graphicx}
\usepackage{amssymb}

\usepackage{appendix}

\usepackage[italian,english]{babel}

\usepackage[utf8x]{inputenc}
\usepackage{fancyhdr}

\usepackage{calc}
\usepackage{url}

\DeclareMathOperator{\Rico}{Ric_o}

\usepackage[text={6.3in,8.6in},centering]{geometry}

\usepackage{srcltx, inputenc}

\usepackage[T1]{fontenc}

\usepackage[usenames,dvipsnames]{color}

\makeatletter
\def\cleardoublepage{\clearpage\if@twoside \ifodd\c@page\else%
         \hbox{}%
     \thispagestyle{empty}
     \newpage%
     \if@twocolumn\hbox{}\newpage\fi\fi\fi}
\makeatother

\theoremstyle{plain}
\newtheorem{theorem}{Theorem}[section]

\newtheorem{lemma}[theorem]{Lemma}

\newtheorem{remark}[theorem]{Remark}

\newtheorem{corollary}[theorem]{Corollary}
\numberwithin{equation}{section}
\theoremstyle{definition}

\begin{document}

\title[]{Blow-up versus global existence of solutions \\for reaction-diffusion equations \\ on classes of Riemannian manifolds}

\author{Gabriele Grillo}
\address{\hbox{\parbox{5.7in}{\medskip\noindent{Dipartimento di Matematica,\\
Politecnico di Milano,\\
   Piazza Leonardo da Vinci 32, 20133 Milano, Italy.
   \\[3pt]
        \em{E-mail address: }{\tt
          gabriele.grillo@polimi.it
          }}}}}

\author{Giulia Meglioli}
\address{\hbox{\parbox{5.7in}{\medskip\noindent{Dipartimento di Matematica,\\
Politecnico di Milano,\\
   Piazza Leonardo da Vinci 32, 20133 Milano, Italy.
   \\[3pt]
        \em{E-mail address: }{\tt
          giulia.meglioli@polimi.it
          }}}}}

\author{Fabio Punzo}
\address{\hbox{\parbox{5.7in}{\medskip\noindent{Dipartimento di Matematica,\\
Politecnico di Milano,\\
   Piazza Leonardo da Vinci 32, 20133 Milano, Italy. \\[3pt]
        \em{E-mail address: }{\tt
          fabio.punzo@polimi.it}}}}}



\maketitle

\maketitle              

\begin{abstract} It is well known from the work of \cite{BPT} that the Fujita phenomenon for reaction-diffusion evolution equations with power nonlinearities does not occur on the hyperbolic space $\mathbb{H}^N$, thus marking a striking difference with the Euclidean situation. We show that, on classes of manifolds  in which the bottom $\Lambda$ of the $L^2$ spectrum of $-\Delta$ is strictly positive (the hyperbolic space being thus included), a different version of the Fujita phenomenon occurs for other kinds of nonlinearities, in which the role of the critical Fujita exponent in the Euclidean case is taken by $\Lambda$. Such nonlinearities are time-independent, in contrast to the ones studied in \cite{BPT}. As a consequence of our results we show that, on a class of manifolds much larger than the case $M=\mathbb{H}^N$ considered in \cite{BPT}, solutions to \eqref{problema} with power nonlinearity $f(u)=u^p$, $p>1$, and corresponding to sufficiently small data, are global in time. Though qualitative similarities with similar problems in bounded, Euclidean domains can be seen in the results, the methods are significantly different because of noncompact setting dealt with.
\end{abstract}

\bigskip
\bigskip

\section{Introduction}

We investigate existence of global in time solutions, versus blow-up in finite time, to nonlinear reaction-diffusion problems of the following type:
\begin{equation}\label{problema}
\begin{cases}
\, u_t= \Delta u +\,f(u) & \text{in}\,\, M\times (0,T) \\
\,\; u =u_0\ge0 &\text{in}\,\, M\times \{0\}\,,
\end{cases}
\end{equation}
where $\Delta$ is the Laplace-Beltrami operator on a Riemannian manifold $M$, 
$T\in (0,\infty]$ and $f:[0,+\infty)\to [0,+\infty)$ is e.g. a locally Lipschitz, increasing function. 
Further specification of the assumptions on the manifold $M$ and on the nonlinearity $f$ will be made later on. A crucial parameter, according to whose value the behaviour of solutions will change, will be the value of $f'(0)$ when $f$ is differentiable in $x=0$. It will be usually required that $u_0\in C(M)\cap L^\infty(M)$ to ensure the existence of classical solutions at least up to a certain time.

The analogue of \eqref{problema} in the Euclidean setting has a long history especially in the particularly important case of power nonlinearities, i.e. for the problem
\begin{equation}\label{problemaeuclideo}
\begin{cases}
\, u_t= \Delta u +\,|u|^{p-1}u & \text{in}\,\, \mathbb{R}^n\times (0,T) \\
\,\; u =u_0 &\text{in}\,\, \mathbb{R}^n\times \{0\}\,,
\end{cases}
\end{equation}
where it is assumed that $u_0\in L^\infty(\mathbb{R}^n)$. It has been shown by Fujita in \cite{F}, and in \cite{H} and \cite{KST} for the critical case, that for problem \eqref{problemaeuclideo} the following facts hold:
\begin{enumerate}[a)]
\item If  $1 < p\le p^*:=\frac{N+2}N$, \eqref{problemaeuclideo} does not possess nontrivial global solutions.
\item If $p > p^*$ solutions corresponding to data that are sufficiently small in a suitable sense, are global in time.
\end{enumerate}

It should be noticed that, by a generalization of a result of Kaplan \cite{K}, solutions corresponding to \it sufficiently large \rm data blow up for any $p>1$. A complete account of results concerning blow-up and global existence of solutions to semilinear parabolic equations posed in $\mathbb R^n$ can be found, e.g., in \cite{BB}, \cite{DL}, \cite{Levine}, \cite{Vaz1} and in references therein.

\medskip

In the case of evolution equation posed on Riemannian manifolds the situation may change completely. In fact, an analogue of \eqref{problemaeuclideo} has been studied  in \cite{BPT} in the important context of the \it hyperbolic space \rm $\mathbb{H}^n$, namely on the simply connected manifold of constant sectional curvature equal to $-1$. It is shown there that for \it all \rm $p>1$ sufficiently small initial data give rise to global in time solutions. A kind of Fujita phenomenon nontheless takes place but when a \it time dependent \rm modification of \eqref{problemaeuclideo} is taken into account. Namely, if the reaction term $u^p$ (considering nonnegative data, hence nonnegative solutions) is replaced by $e^{\alpha t}u^p$, $\alpha>0$ being a fixed parameter, then a Fujita-type phenomenon then takes place, the threshold value being $p^\sharp:=1+\frac\alpha\Lambda$ where $\Lambda:=(N-1)^2/4$ is the bottom of the $L^2$ spectrum of $-\Delta$ on $\mathbb{H}^n$. Informally, one might say that for this to hold the nonlinearity must be amplified exponentially as time grows, in fact it is also shown in \cite{BPT} that if the exponential factor in time is replaced by a power of time the Fujita phenomenon still does not occur. See also \cite{WY} for a careful analysis of the critical case $p=p^\sharp$, in which the authors show, combining their results with the ones of \cite{BPT}, that in such case global in time solutions exist for all values of $\alpha$, thus marking a further difference with the Euclidean case. Further precise results on the lifespan of solutions that are not global are given in \cite{WY2}.

Recently, a number of results concerning blow-up and global existence for solutions of nonlinear parabolic equations with power-like reaction term and \it nonlinear, slow diffusion of porous medium type \rm has also been obtained, on some classes of Riemannian manifolds, in \cite{MMP}, \cite{MeGP1}, \cite{MeGP2}, \cite{Sun}, \cite{Z}. The results for this class of equations show usually several differences with the case involving a linear diffusion, already in the Euclidean case, see \cite{gala}.

\medskip
Our goal here will be to present some result for different type of \it time independent \rm nonlinearities in which a new kind of Fujita phenomenon takes place, in a wide class of manifolds that includes the hyperbolic space. We shall consider noncompact complete Riemannian manifold  $M$ of infinite volume, with dimension $N\ge 3$, and require that some further conditions on $M$, to be described below, hold. In our first result, Theorem \ref{teo1}), we require the following additional condition on $M$:

\begin{itemize}
\item $M$ is stochastically complete, and
$\lambda_1(M):= \inf \operatorname{spec}(-\Delta)\,\,>0$, where $\operatorname{spec}(-\Delta)$ is the $L^2$ spectrum of the Laplace-Beltrami operator $-\Delta$.
\end{itemize}

Under such assumptions, we prove \it nonexistence of global solutions for problem \eqref{problema} e.g. if $f$ is convex and increasing,  $f'(0)>\lambda_1(M)$, and $1/f$ is integrable at infinity\rm.

\medskip

Stochastic completeness is a well-studied property of Riemannian manifolds, and it amounts to requiring that $T_t1=1$ for all $t>0$, or equivalently that
\begin{equation}\label{stoc}
\int_M p(x,y,t) \, d\mu(y)\,= \,1,\quad \text{for all}\,\,\, x\in M,\,\, t>0.
\end{equation}
where $p(x,y,t)$ is the heat kernel of the manifold $M$ and $\mu$ the Riemannian measure. See e.g. \cite{Grig, Grig3} for a number of conditions on $M$ ensuring that stochastic completeness holds. For example, it suffices that, for some $o\in M$, the function $ r \mapsto \frac{ r}{\log \mathcal{V}(o,r)}$ is not integrable at infinity, where $\mathcal{V}(o,r)$ is the volume of the geodesic ball of radius $r$ centered at $o$. Note that this is true in particular if $\mathcal{V}(o,r)\le C e^{ar^2}$ for suitable $C,a>0$. This allows e.g. sectional curvatures to tend to $-\infty$ at the infinity of $M$ (namely when $\varrho(o,x)\to+\infty$, $\varrho$ being the Riemannian distance and $o$ being fixed in $M$), at most quadratically, in a suitable precise sense, see also \cite{GIM} for relations to nonlinear elliptic and parabolic equations on $M$. As for the assumption $\lambda_1(M)\,>0$ we comment that a well-known sufficient condition for this to hold is that $\mathrm{sec}\, \le c<0$, $\mathrm{sec}$ denoting sectional curvatures. Thus, the class of manifolds on which the above result works is large, as it includes e.g. all those manifolds whose sectional curvatures are pinched between two strictly negative constants, and in particular the hyperbolic space.

\medskip

In our second result, Theorem \ref{teo2}, the additional assumption we require on $M$ beside the previous ones is the following:
\begin{itemize}
\item the following Faber-Krahn inequality holds:
for some $c>0$, for any non-empty relatively compact open subset $\Omega\subset M$,
\begin{equation}\label{FK}\lambda_1(\Omega)\geq \frac{c}{[\mu(\Omega)]^{\frac 2N}},\end{equation}
where $\lambda_1(\Omega)$ is the first eingevalue of the Laplace operator on $\Omega$ completed with homogeneous Dirichlet boundary conditions, and $\mu$ is the Riemannian measure.
\end{itemize}

Under such assumptions, we show \it existence of global solutions for small data e.g. if $f$ is differentiable at $x=0$ with $f'(0)< \lambda_1(M)$\rm.

\medskip

We refer e.g. to \cite[Cor. 14.23, Cor. 15.17]{Grig3} for equivalent conditions for the Faber-Krahn inequality to hold. In particular, its validity is implied by on-diagonal bounds for the heat kernel in the form $p(t,x,x)\le c\,t^{-N/2}$ for all $x\in M$, $t>0$, or by the validity of the Sobolev (or the Nash) inequality on $M$, which is in turn satisfied e.g. when sec$\,\le0$ on $M$, sec denoting sectional curvatures.

\medskip

Observe that our blow-up result is similar in character to the well-known blow-up result for bounded domains of $\mathbb R^n$ (see, e.g., \cite[Section 3.2]{BB}). However, the methods of proof exploited for bounded domains do not work on general Riemannian manifolds, thus our arguments are completely different. Indeed, in bounded domains the blow-up result is usually obtained by means of the Kaplan method (see \cite{BB, K}), which makes use of the first eigenvalue and of the first eigenfunction of the Laplacian. In order to extend that argument to a general Riemannian manifold, it would be necessary to know precisely the behaviour at infinity of the positive solution $\phi$ to
\[\Delta \phi+\lambda_1(M) \phi = 0 \quad \text{ in }\, M\,\]
and in particular its integrability properties w.r.t. the Riemannian measure, which are not known in general, e.g. on the hyperbolic space $\mathbb H^n$ $\phi$ belongs just to $L^{2+\varepsilon}$ for all $\varepsilon>0$.  The Kaplan method yields indeed partial results on the subclass of Cartan-Hadamard manifolds, i.e. simply connected Riemannian manifolds with nonpositive sectional curvatures, with Ricci curvature bounded from below, see \cite{Pu22}  when $f(u)=u^p$, but the method can not be pushed to get the sharp threshold value $\lambda_1(M)$ provided in Theorem \ref{teo1}, \ref{teo2}, even in the special case of Cartan-Hadamard manifolds and, in fact, even on the special case $M=\mathbb H^n$.


Let us mention that some blow-up results in bounded domains have also been established in \cite{M1}, \cite{M2}, for more general operators, by means of the method of sub-- and supersolutions. Those results seem to be quite implicit in character, and they are based on the asymptotic behaviour for large times of solutions to the associated linear problem.

\normalcolor

\medskip

As a further comment, we mention that it is easy to show that, on a wide class of manifolds, characterized by the validity of the parabolic maximum principle, a sufficient condition for which is e.g. the very general condition \eqref{bound2} below, \it blow-up of solutions corresponding to large data occurs\rm, provided the nonlinearity $f$ is convex and $1/f$ is integrable at infinity\rm. See the end of this section for some more detail and Section \ref{fr} for a concise proof.

\medskip

Let us now give briefly some more precise detail on the conditions on the nonlinearities to be verified in order to prove our results. In the first one, Theorem \ref{teo1}, we shall assume that $f$ is continuous. Besides, we assume that $f\ge h$ where $h$ is increasing and convex in $[0,+\infty)$, that it satisfies $\int^{+\infty}\frac{1}{h(s)}\,ds\,<+\infty$, that $h(0)=0$ and finally that the condition $h'(0)>\lambda_1(M)$ holds. Of course, a sufficient condition for this to hold is that the conditions satisfied by $h$ are satisfied by the nonlinearity $f$ itself, as mentioned above. Then we show that \it all \rm solutions blow up in finite time. In Theorem \ref{teo2} we shall show that, for any locally Lipschitz nonlinearity $f$ such that $f(x)\le \lambda x$ in a neighbourhood of $x=0$, with $\lambda \le\lambda_1(M)$, then \it sufficiently small \rm data give rise to solutions existing for all times. Of course, if $f'(0)$ exists, a sufficient condition for the above condition on $f$ to hold for some $\lambda<\lambda_1(M)$ is $f'(0)<\lambda_1(M)$, as mentioned above. The combination of the two results thus shows the version of the Fujita phenomenon we aim at.

\medskip
It is important to mention that the above results provide, as immediate consequences, new results w.r.t. the ones proved in \cite{BPT} for \eqref{problema} even in the classical case $f(u)=u^p$. In fact, Theorem \ref{teo2} shows in particular that \it solutions corresponding to sufficiently small data are global\rm, on a much wider class of manifolds than the hyperbolic space considered in one of the main result of \cite{BPT}, see Corollary \ref{small} for a precise statement. In particular the results holds e.g. on all those manifolds whose sectional curvatures are pinched between two strictly negative constants everywhere.

\medskip

%


The paper is organized as follows. In Section \ref{prel} we collect some preliminary material on the class of manifolds considered and on the concept of solution.
In Section \ref{statements} we state all our main results.   Section \ref{proofs} contains the proof of Theorem \ref{teo1}, whereas Theorem \ref{teo2} is proved in Section \ref{existence}.  In Section \ref{fr} we complement our main results by considering a large class of manifolds, e.g. those ones in which the radial Ricci curvature does not diverge at infinity faster that $-cr^2$, where $r$ is the Riemannian distance from a given pole $o\in M$. We give a concise proof of the fact that on such manifolds, if $f$ is a locally Lipschitz, increasing, convex function such that $1/f$ is integrable at infinity, large data give rise to solutions blowing up in finite time.

\section{Preliminaries}\label{prel}

\subsection{Heat semigroup on $M$}
Let $\{e^{t\Delta}\}_{t\ge 0}$ be the heat semigroup of $M$, acting on $L^p(M)$ for all $p\in[1,+\infty]$. It admits a (minimal) \textit{heat kernel}, namely a function $p\in C^{\infty}(M\times M\times (0,+\infty))$, $p>0$ in $M\times M\times (0,+\infty)$ such that
\begin{equation*}
(e^{t\Delta} u_0)(x)=\int_M p(x,y,t)\,u_0(y) \, d\mu(y), \quad x\in M,\,\, t>0,
\end{equation*}
for any $u_0\in L^p(M)$. It is well known that
\begin{equation}\label{eq23}
\int_M p(x,y,t) \, d\mu(y)\,\le \,1,\quad \text{for all}\,\,\, x\in M,\,\, t>0.
\end{equation}

As recalled in the Introduction, we say that a manifold $M$ is stochastically complete if the following condition holds:
\begin{equation*}
\int_M p(x,y,t) \, d\mu(y)\,= \,1,\quad \text{for all}\,\,\, x\in M,\,\, t>0.
\end{equation*}
See the considerations and the references after \eqref{stoc} for sufficient conditions for this fact to hold.


Furthermore, it is known that if $M$ is a noncompact Riemannian manifold, then (see \cite[Corollary 1]{CK})
\begin{equation}\label{eq26}
\lim_{t\to+\infty} \frac{\log p(x,y,t)}{t}\,=\, -\lambda_1(M)\quad \text{locally uniformly in}\,\,\, M\times M\,,
\end{equation}
where $\lambda_1(M)$ is the infimum of the $L^2$ spectrum of $-\Delta$.
We also recall that from the Faber-Krahn inequality \eqref{FK} it follows that (see \cite[Cor. 15.17 (b)]{Grig3}), for some $\bar C>0$:
\begin{equation}\label{eq26a}
p(x,y,t)\,\le\, \bar C\,e^{-\lambda_1\,t}\quad \text{for any}\,\,\, x,y\in M, t\geq 1.
\end{equation}


\subsection{On the concept of solution}

We shall always deal with bounded initial data. Solutions will be meant in the classical sense. More precisely, setting $Q_T=M\times(0,T]$, we require that $u\in C^{2,1}(Q_T)\cap C(\overline{Q_T})\cap L^\infty(\overline{Q_T})$ and that \eqref{problema} holds in the classical sense.

We shall use in the sequel two different concepts of solution. On the one hand
a function $u\in C(M\times (0,\tau])\cap L^{\infty}(M\times (0,\tau])$, for every $\tau\in (0,T]$ is called a \textit{mild solution} of problem \eqref{problema} if
\begin{equation}\label{eq27}
u(x,t)=(e^{t\Delta}u_0)(x) + \int_0^t \left(e^{(t-s)\Delta}f(u)\right)(x)\,ds\,,
\end{equation}
for any $t\in [0,\tau]$.


We notice that, by adapting the methods of \cite[Prop. 2.1, Lemma 2.1]{BPT}, for \it bounded \rm initial data $u_0$ and up to a time $T$ such that $u(t)$ is bounded for all  $t\in[0,T)$ (blow-up might occur at some positive time), the two concepts of solutions coincide provided $f$ is locally Lipschitz, as required in our main results. Hence we shall use them indifferently when needed.

\section{Statements of main results}\label{statements}
In this section, we state our results concerning solutions to problem \eqref{problema}.
%
We say that a solution {\em blows up in finite time}, whenever there exists $\tau>0$ such that
\[\lim_{t\to \tau^-}\|u(t)\|_{L^\infty(M)}=+\infty\,.\] Otherwise, if a solution $u(t)\in L^\infty(M)$ for all $t>0$, we say that it is {\em global}. Our first results involves nonexistence of global solutions. Notice that the assumptions on the auxiliary function $h$ below entail that $h$ is differentiable in $x=0$.
\begin{theorem}\label{teo1}
Let $M$ be a complete, non compact, stochastically complete Riemannian manifold with $\lambda_1(M)>0$. Let $u_0\in C(M)\cap L^\infty(M), u_0\ge 0, u_0\not\equiv 0 \text{ in } M$. Let $f$ be locally Lipschitz in $[0, +\infty)$. Assume that $f\ge h$ where $h$ is increasing and convex in $[0,+\infty)$ and $h(0)=0$. Moreover, suppose that
\begin{equation}\label{N5}
\int^{+\infty}\frac{1}{h(s)}\,ds\,<+\infty,
\end{equation}
and finally that $h'(0)>\lambda_1(M)$.
Then any solution to problem \eqref{problema} blows up in finite time.
\end{theorem}

Notice that in the above Theorem the fact that $h$ is assumed to be increasing and convex implies the existence of $h'(0)$.

\begin{theorem}\label{teo2}
Let $M$ be a complete, non compact, stochastically complete Riemannian manifold with $\lambda_1(M)>0$ and such that the Faber-Krahn inequality \eqref{FK} holds. Assume also that $f$ is increasing, locally Lipschitz and $f(0)=0$. Moreover, suppose that, for some $\delta>0$ and $0<\alpha\leq \lambda_1(M)$,
\begin{equation}\label{N1}
f(x)\le \alpha x\quad \text{ for all }\,\, x\in [0, \delta]\,.
\end{equation}

Furthermore, assume that $u_0\in C(M)\cap L^\infty(M)\cap L^1(M), u_0\ge 0 \text{ in } M$ is small enough.
Then there exists a global solution to problem \eqref{problema}; in addition, $u\in L^{\infty}(M\times(0,+\infty))$.
\end{theorem}

The smallness condition on $u_0$ in Theorem \ref{teo2} can be precisely formulated. Indeed, our hypothesis is that

\begin{equation}\label{eq12a}
\|u_0\|_{L^{\infty}(M)}\,\le\,\delta\,e^{-\alpha},
\end{equation}
and
\begin{equation}\label{eq13a}
\|u_0\|_{L^{1}(M)}\,\le\,\frac{\delta}{C_2},
\end{equation}
where $C_2=\bar C\,e^{-(\lambda_1-\alpha)}$, $\bar C$ being defined in \eqref{eq26a}, while $\alpha$ and $\delta$ are given by \eqref{N1}.

As a consequence, we can generalize one of the main results of \cite{BPT} (see also \cite{Pu3}) to a class of manifolds much wider than $\mathbb{H}^n$.

\begin{corollary}\label{small}
             Let $M$ be a complete, non compact, stochastically complete Riemannian manifold with $\lambda_1(M)>0$ and such that the Faber-Krahn inequality \eqref{FK} holds. Assume $f(x)=x^p$ for all $x\ge0$ with $p>1$. Assume also that $u_0\in C(M)\cap L^\infty(M)\cap L^1(M), u_0\ge 0 \text{ in } M$ is small enough.
Then there exists a global solution to problem \eqref{problema}; in addition, $u\in L^{\infty}(M\times(0,+\infty))$.
\end{corollary}

\begin{remark}

\begin{itemize}
\item \rm For any $p>1$, let
\[f(u)=\begin{cases}
\alpha u, & u\in [0, 1],\\
\alpha u^p & u\in (1, +\infty).
\end{cases}\]
If $\alpha>\lambda_1(M)$, then by Theorem \ref{teo1}, the solution to problem \eqref{problema} blows up in finite time for any nontrivial $u_0$. On the other hand, if $\alpha\leq \lambda_1(M)$, then the solution exists globally in time, provided that $u_0$ is sufficiently small.

\item Let $f(u)=e^{\beta u}-1$ with $\beta>0$. By Theorem \ref{teo1}, if $\beta>\lambda_1(M)$, then the solution to problem \eqref{problema} blows up in finite time. On the contrary, if $\beta<\lambda_1(M)$, then condition \eqref{N1} is satisfied with $\beta<\alpha\leq \lambda_1(M)$. Therefore, by Theorem \ref{teo2}, the solution exists globally in time, whenever $u_0$ is sufficiently small.

    \end{itemize}
\end{remark}

By standard methods, on suitable manifolds and for a wide class of nonlinearities $f$, it is possible to show that whenever $u_0$ is large enough, blow-up of solutions occurs. We defer the discussion of this fact to Section \ref{fr}.

\section{Finite time blow-up for any initial datum}\label{proofs}

\subsection{Two key estimates}

Let us first prove a preliminary lemma.

\begin{lemma}\label{lemma1}
Let $M$ be a complete, non compact Riemannian manifold with $\lambda_1(M)>0$. Let $u_0\in C(M)\cap L^\infty(M), u_0\ge 0, u_0\not\equiv 0 \text{ in } M$.
Let $\varepsilon \in (0,\lambda_1(M))$. Then there exist $\Omega\subset M$, $t_0>0$, $C_1>0$ such that
\begin{equation}\label{eq31}
(e^{t\Delta}u_0)(x) \ge \, C_1 e^{-[\lambda_1(M)+\varepsilon]t}\,, \quad \text{for any}\,\,\, x\in\Omega,\,\,\,t>t_0\,.
\end{equation}
\end{lemma}

\begin{proof}
Let $\Omega \subset M$ be such that $\mu(\Omega)<+\infty$, $\int_{\Omega} u_0\, d\mu>0$. From \eqref{eq26}, there exists $t_0>0$ such that, for every $x,y\in\Omega$,
$$
p(x,y,t)\ge e^{-[\lambda_1(M)+\varepsilon]t} \quad \text{for every}\, t>t_0\,.
$$
Hence
$$
\begin{aligned}
(e^{t\Delta}u_0)(x) &\ge \int_M p(x,y,t) u_0(y)\, d\mu(y) \\
&\ge e^{-[\lambda_1(M)+\varepsilon]t}\, \int_{\Omega}u_0(y) \,d\mu(y).
\end{aligned}
$$
Consequently, we obtain \eqref{eq31} with $C_1:= \int_{\Omega}u_0(y) \,d\mu(y)>0\,.$
\end{proof}

Let $u$ be a mild solution of equation \eqref{problema}, so that it fulfills \eqref{eq27}. Then, for any $x\in M$ and for any $T>0$, we define
\begin{equation}\label{eq32}
\Phi_x^T(t)\equiv \Phi_x(t):=\int_M p(x,z,T-t)\, u(z,t)\, d\mu(z)\,\quad \text{ for any }\,\, t\in [0, T]\,.
\end{equation}
Observe that
\begin{equation}\label{eq33}
\Phi_x(0)=\int_M p(x,z,T) \,u_0(z)\, d\mu(z)\,= (e^{T\Delta}u_0)(x), \,\, x\in M\,.
\end{equation}
Suppose that $u_0\in L^{\infty}(M)$. Choose any
\begin{equation}\label{eq34a}
\delta>\|u_0\|_{L^{\infty}(M)}.
\end{equation}
From \eqref{eq34a} and \eqref{eq23} we obtain that, for any $x\in M$,
\begin{equation}\label{eq34b}
\begin{aligned}
\Phi_x(0)&=\int_M p(x,z,T) \,u_0(z)\, d\mu(z)\\
&\le\,\|u_0\|_{L^{\infty}(M)}\int_M p(x,z,T)\, d\mu(z)\\
&< \delta.
\end{aligned}
\end{equation}

We now state the following lemma.

\begin{lemma}\label{lemma2}
Let $M, f, h, u_0$ be as in Theorem \ref{teo1}. Let $x\in M$ and $\Phi_x(t)$ be as in \eqref{eq32}. Set $\alpha:=h'(0).$ Then
\begin{equation}\label{eq34}
\Phi_x(0)\,\,\le\,\, C\,e^{-\alpha T}, \quad \text{for any}\,\,T\ge \bar t,
\end{equation}
for suitable $\bar t>0$ and $C>0$, depending on $x$.
\end{lemma}
Note that $\bar t$ and $C$ are given by \eqref{eq312a} and \eqref{eqf2} below, respectively, with $\delta$ as in \eqref{eq34a}.

\begin{proof}
Let $u$ be a solution to problem \eqref{problema}. So \eqref{eq27} holds; hence
\begin{equation}\label{eq30f}
u(z,t)=\int_M p(z,y,t) u_0(y)\,d\mu(y) + \int_0^t\int_{M} p(z,y,t-s)f(u)\,d\mu(y)\, ds.
\end{equation}
In the definition of $\Phi_x^T(t)\equiv \Phi_x(t)$ (see \eqref{eq32}) fix any
\begin{equation}\label{eq20f}
T>\frac1{\alpha}\left[\log\delta -\log\Phi_x(0)\right]\,.
\end{equation}
We multiply \eqref{eq30f} by $p(x,z,T-t)$ and  integrate over $M$. Therefore, we get
\begin{equation}\label{eq35}
\begin{aligned}
\int_M p(x,z,T-t)\,u(z,t)\,d\mu(z) &=\int_M\int_M p(z,y,t)\,u_0(y)\,p(x,z,T-t)\,d\mu(z)\,d\mu(y)\\
&+ \int_0^t\int_{M}\int_M p(z,y,t-s)\,f(u)\,p(x,z,T-t)\,d\mu(z)d\mu(y)ds.
\end{aligned}
\end{equation}
Now, due to \eqref{eq32}, for all $t\in (0, T),$ equality \eqref{eq35} reads
$$
\begin{aligned}
\Phi_x(t)&=\int_M\int_M p(z,y,t)\,u_0(y)\,p(x,z,T-t)\,d\mu(z)\,d\mu(y) \\
&+ \int_0^t\int_{M}\int_M p(z,y,t-s)\,f(u)\,p(x,z,T-t)\,d\mu(z)\,d\mu(y)\,ds.
\end{aligned}
$$
By \eqref{eq33}, for all $t\in (0, T),$
\begin{equation*}
\begin{aligned}
\Phi_x(t)&= \int_M p(x,y,T)\,u_0(y)\,d\mu(y)+\int_0^t\int_{M} f(u)\,p(x,y,T-s)\,d\mu(y)\,ds \\
&= \Phi_x(0)+\int_0^t\int_{M} f(u)\,p(x,y,T-s)\,d\mu(y)\,ds\,.
\end{aligned}
\end{equation*}
Since $f\geq h$ in $[0, +\infty$),
\begin{equation}\label{eq35b}
\begin{aligned}
\Phi_x'(t)&= \int_{M} f(u)\,p(x,y,T-t)\,d\mu(y) \\
&\ge \int_{M} h(u)\,p(x,y,T-t)\,d\mu(y)\,.
\end{aligned}
\end{equation}

Since $h$ is an increasing convex function, due to \eqref{stoc}, by using Jensen inequality, we get
\begin{equation}\label{eq37}
\int_M p(x,y,T-t)\,h\left(u(y,t)\right)\,d\mu(y)\geq h\left(\int_M p(x,y,T-t)\,u(y,t)\,d\mu(y)\right)=h(\Phi_x(t))\,.
\end{equation}

%

Combining together \eqref{eq35b} and \eqref{eq37}, we obtain
\begin{equation}\label{eq311}
\Phi_x'(t)\,\ge \, h(\Phi_x(t))\quad \text{for all }\,\, t\in (0, T)\,.
\end{equation}

Fix any $x\in M$. We first observe that \eqref{eq311} implies that $\Phi_x(t)$ is an increasing function w.r.t the time variable $t$, since
\begin{equation}\label{eq311c}
\Phi_x'(t)\,>\,0\quad \text{for any}\,\,t\in (0, T).
\end{equation}
Moreover, due to  \eqref{eq34a} and  \eqref{eq34b}, by continuity of $t\mapsto \Phi_x(t)$, we can infer that there exists $t_1>0$ such that
\begin{equation}\label{eqf1}
\Phi_x(t)<\delta \quad \text{for all }\,\, t\in (0, t_1)\,.
\end{equation}

Since $h$ is convex, increasing in $[0, +\infty), h(0)=0, h'(0)=\alpha$, then
\begin{equation}\label{N3}
h(s)\geq \alpha s\quad \text{ for all }\,\, s\geq 0\,.
\end{equation}
Due to \eqref{N3} and to \eqref{eqf1}, we get
\begin{equation*}
\begin{cases}
&\Phi_x'(t)\,\ge\, \alpha\, \Phi_x(t) \quad \text{for any}\,\,t\in(0,t_1), \\
&\Phi_x(0)\,<\,\delta.
\end{cases}
\end{equation*}
Let
\[\bar t:=\sup\{t>0\,:\, \Phi_x(t)<\delta\}\,.\]
We claim that
\begin{equation}\label{eq320}
0<\bar t\,\le \, -\frac{1}{\alpha}\log(\Phi_x(0)) +\frac{\log \delta}{\alpha}\,.
\end{equation}

In order to show \eqref{eq320}, consider the Cauchy problem
\begin{equation*}
\begin{cases}
y'(t) = \alpha y(t),& t>0\\
y(0)=\Phi_x(0)\,.
\end{cases}
\end{equation*}
Clearly,
\[y(t)=\Phi_x(0) e^{\alpha t}, \quad t>0\,.\]
Hence
\begin{equation*}
y(\tau)=\delta \quad \text{whenever }\,\, \tau=\frac 1{\alpha}\left[\log(\delta)-\log(\Phi_x(0))\right]\,.
\end{equation*}
Furthermore, note that, in view of \eqref{eq34b} and \eqref{eq20f},
\begin{equation}\label{eq31f}
0<\tau <T\,.
\end{equation}
By comparison,
\[\Phi_x(t)\geq y(t)\quad \text{for all }\, t\in (0, T)\,.\]
Thus, we can infer that there exists $\bar t\in (0, \tau\,]$ such that
\begin{equation}\label{eq312a}
\Phi_x(\bar t)=\delta.
\end{equation}
In particular, from \eqref{eq31f} it follows that $\bar t<T\,.$
\smallskip

Due to \eqref{eq311c} and \eqref{eq312a}, we obtain that
\begin{equation}\label{eq312}
\Phi_x(t)\,>\,\delta\quad \text{for any}\,\,\, \bar t<t <T.
\end{equation}
By \eqref{eq311}, in particular we have
\begin{equation}\label{eq314}
\Phi_x'(t)\,\ge \,h(\Phi_x(t))\quad \text{for any}\,\,\,\bar{t}<t<T.
\end{equation}

Define
\begin{equation*}
G(t):=\int_{\Phi_x(t)}^{+\infty} \frac{1}{h(z)}\,dz\,\quad \text{for all }\,\, \bar t<t <T \,.
\end{equation*}
Note that $G$ is well-defined thanks to hypothesis \eqref{N5} and to \eqref{eq312}. Furthermore,
\begin{equation}\label{eq39}
G'(t)=-\frac{\Phi_x'(t)}{h(\Phi_x(t))}\,\quad \text{for any }\,\, \bar t<t<T\,.
\end{equation}
We now define
\begin{equation}\label{eq315}
w(t):=\exp\{G(t)\}\quad \text{for any}\,\,\,\bar{t}<t<T.
\end{equation}
Then, due to \eqref{eq39} and \eqref{eq314},
\begin{equation}\label{eq39b}
\begin{aligned}
w'(t)&=-\frac{\Phi_x'(t)}{h(\Phi_x(t))}w(t)\\
&\le - \,\,w(t)\,\quad\quad\quad \text{for any}\,\,\,\bar{t}<t<T.
\end{aligned}
\end{equation}
By integrating \eqref{eq39b} we get:
\begin{equation}\label{eq316}
\begin{aligned}
w(t)&\le w(\bar t\,)\exp\left\{-\int_{\bar t}^t\,ds\right\}\\
&\le w(\bar t\,)\exp\left\{-\,(t-\bar t)\right\} \quad \text{for any}\,\,\,\bar{t}<t<T.
\end{aligned}
\end{equation}
We substitute \eqref{eq315} into \eqref{eq316}, so we have
\begin{equation*}
\exp\{G(t)\}\,\le\,\exp\left\{G(\bar t)-\,(t-\bar t)\right\}\quad \text{for any}\,\,\,\bar{t}<t<T.
\end{equation*}
Thus, for any $\bar t<t<T$,
\begin{equation}\label{eq319}
G(t)\,\le\,G(\bar t)-\,(t-\bar t\,).
\end{equation}

We now combine \eqref{eq319} together with \eqref{eq320}, hence
$$
0\le G(t)\,\le\,G(\bar t)-\,\left(t +\frac{1}{\alpha}\log(\Phi_x(0)) -\frac{\log\delta}{\alpha}\right) \quad  \text{for any} \,\,\, \bar t<t<T.
$$
Hence
\begin{equation}\label{eq325}
\log(\Phi_x(0))\,\le \, \alpha G(\bar t)+\log\delta -\alpha \,\,t\, \quad  \text{for any} \,\,\, \bar t<t<T.
\end{equation}
We now take the exponential of both sides of \eqref{eq325}. Thus we get, for any $\bar t<t<T$,
$$
\begin{aligned}
\Phi_x(0)&\le \exp\left\{\alpha G(\bar t)+\log\delta -\alpha \,\,t\right\}\\
&= C\,\exp\{-\alpha\,t\}\,,
\end{aligned}
$$
where
\begin{equation}\label{eqf2}
C:=\exp\{\alpha G(\bar t)+\log\delta\}\,.
\end{equation}
This is the inequality \eqref{eq34}.
\end{proof}

\subsection{Proof of Theorem \ref{teo1}}
\begin{proof}
Take any $\delta>0$ fulfilling \eqref{eq34a}.  We suppose, by contradiction, that $u$ is a global solution of problem \eqref{problema}.
Since $\alpha:=h'(0)>\lambda_1(M)$, there exists $\varepsilon\in (0,\alpha-\lambda_1(M))$ such that
$$
\alpha>\lambda_1(M)+\varepsilon.
$$
Let $\Omega \subset M$ be such that
$$
\mu(\Omega)<+\infty \quad \text{and}\quad \int_{\Omega} u_0 \,  d\mu >0.
$$
Then
\begin{equation}\label{eq327}
\lim_{T\to +\infty} e^{\alpha-(\lambda_1(M)+\varepsilon)]T}=+\infty.
\end{equation}
Fix any arbitrary $x\in M$. By Lemma \ref{lemma1} and Lemma \ref{lemma2},
$$
C_1 \, e^{-[\lambda_1(M)+\varepsilon]\,T}\,\,\le\,\,(e^{T\Delta}u_0)(x)\,\,\le\,\, C\,e^{-\alpha\, T}, \quad \text{for any}\,\,\, T>\max\{t_0, \bar t\,\},
$$
where $t_0>0$, $C_1>0$ are given in Lemma \ref{lemma1}, while $\bar t>0$, $C>0$ in Lemma \ref{lemma2}. Hence, if $u$ exists globally in time, we would have
\begin{equation}\label{eq326}
e^{[\alpha-(\lambda_1(M)+\varepsilon)]\,T}\,\,\le\,\, \frac{C}{C_1} \quad \text{for any}\,\,\, T>\max\{t_0, \bar t\,\}.
\end{equation}
Nonetheless, due to \eqref{eq327}, the left hand side of \eqref{eq326} tends to $+\infty$ as $T\to \infty$.
Thus, we have a contradiction. Hence the thesis follows.


\end{proof}

\section{Global existence}\label{existence}

Consider the linear Cauchy problem for the heat equation
\begin{equation}\label{eq51}
\begin{cases}
&v_t=\Delta v\quad\,\, \text{in}\,\,\, M\times(0,+\infty)\\
&v=u_0\quad\quad \text{in}\,\,\,M\times\{0\},
\end{cases}
\end{equation}
with $u_0$ as in Theorem \ref{teo2}.
Observe that problem \eqref{eq51} admits the classical solution
\begin{equation}\label{eq52}
v(x,t)=\int_M p(x,y,t)\,u_0(y)\,dy, \quad x\in M, t\geq 0.
\end{equation}
Hence, since $u_0\in L^\infty(M)$,
\begin{equation}\label{eq55}
\|v(t)\|_{L^{\infty}(M)}\le\|u_0\|_{L^{\infty}(M)}\quad \text{for any}\,\,\,t>0\,.
\end{equation}

Moreover, since $u_0\in L^1(M)$, if the Faber-Krahn inequality holds, then, due to \eqref{eq26a},
\begin{align}
v(x,t) \le \bar{C}\, \|u_0\|_{L^1(M)}\,e^{-\lambda_1 t}\quad \text{for any}\,\,\,x\in M, t>1, \label{eq54}
\end{align}
where $\bar C$ has been defined in \eqref{eq26a}.

\smallskip

Let $\{\Omega_j\}_{j\in\mathbb{N}}\subset M$ be a sequence of domains such that
$$
\begin{aligned}
&\Omega_j\subset \Omega_{j+1}\quad \text{for any}\,\,\,j\in \mathbb{N},\\
&\bigcup_{j\in \mathbb N} \Omega_j=M,\\
&\partial\Omega_j\,\,\,\text{is smooth for every}\,\,\,j\in\mathbb{N}.
\end{aligned}
$$
Furthermore, for every $j\in\mathbb{N}$ let $\zeta_j\in C_c^{\infty}(\Omega_j)$ be such that $0\le \zeta_j\le 1$, $\zeta_j\equiv 1$ in $\Omega_{j/2}$.

\begin{proof}[Proof of Theorem \ref{teo2}] We consider initial data $u_0$ satisfying \eqref{eq12a} and \eqref{eq13a}.
Define
\begin{equation*}
\bar{u}(x,t):= e^{\alpha \,t}\,v(x,t),\quad x\in M, t\geq 0,
\end{equation*}
with $v$ and $\alpha$ given by \eqref{eq52} and \eqref{N1}, respectively.

Note that, due to \eqref{eq12a} and \eqref{eq55}, for any $x\in M, t\in(0,1]$,
\begin{equation}\label{eq59}
0\leq \bar u(x,t)\leq e^{\alpha\,t}\|v(t)\|_{L^{\infty}(M)}\,\le\, e^{\alpha\, t}\|u_0\|_{L^{\infty}(M)}\,\le\,\delta.
\end{equation}
Moreover, due to \eqref{eq13a}, \eqref{eq54}, since $\alpha\leq \lambda_1(M)$, for any $t>1$ we get
\begin{equation}\label{eq510}
0\leq \bar u(x,t)\leq e^{\alpha\,t}\|v(t)\|_{L^{\infty}(M)}\le\, \bar{C}\,\|u_0\|_{L^1(M)}e^{-(\lambda_1-\alpha)\, t}\le \,\delta.
\end{equation}
Inequalities \eqref{eq59} and \eqref{eq510} yield
\begin{equation}\label{eq58}
0\,\le\,\bar{u}(x,t)\,\le\, \delta\quad \text{for any}\,\,\,x\in M,\,\, t>0.
\end{equation}
Furthermore, we have
$$
\bar{u}_t-\Delta\bar{u}-f(\bar u)=\alpha\,e^{\alpha\,t}\,v+e^{\alpha\,t}\,v_t-e^{\alpha\,t}\Delta v-f(\bar u).
$$
Now, by using the fact that $v$ is a classical solution to problem \eqref{eq51}, due to \eqref{N1} and \eqref{eq58}, we get
\begin{equation}\label{eq511}
\bar{u}_t-\Delta\bar{u}-f(\bar u)\ge \alpha \bar u - f(\bar u) \geq \,0.
\end{equation}
Hence $\bar u$ is a weak supersolution to problem \eqref{problema} in $M\times(0,\infty)$.

For any $j\in \mathbb{N}$ there exists a unique classical solution $u_j$ to problem
\begin{equation}\label{eq56}
\begin{cases}
& \partial_t u=\Delta u + f(u)\quad \text{in}\,\,\,\Omega_j\times (0,T)\\
& u=0\quad\quad\quad\quad \quad \quad \text{in}\,\,\,\partial\Omega_j\times (0,T)\\
&u=\zeta_j\,u_0 \quad\quad\quad \quad\,\,\, \text{in}\,\,\,\Omega_j\times \{0\}\,.
\end{cases}
\end{equation}
Clearly, $u_j\not\equiv 0$ because $u_0\,\zeta_j\not\equiv 0$ in $\Omega_j$. Moreover, for any $j\in \mathbb{N}$, in view of \eqref{eq511}, since
\[ v= u_0\geq \zeta_j u_0 \quad \text{ in }\,\, M\times \{0\},\]
$\bar u$ is a bounded weak supersolution of problem \eqref{eq56}. Obviously, for any $j\in \mathbb{N}$, $\underline u\equiv 0$ is a subsolution to problem \eqref{eq56}.
Hence, by the comparison principle, for every $j\in\mathbb{N}$ we obtain
\begin{equation}\label{eq512}
0\,\le\,u_j\,\le\,\bar{u}\quad \text{for any}\,\,\,(x,t)\in \Omega_j\times(0,+\infty).
\end{equation}
By standard a priori estimates (see, e.g.,  \cite[Chapter 5]{LSU}), we can infer that there exists a subsequence $\{u_{j_{k}}\}$ of $\{u_{j}\}$, which converges in $C^{2,1}_{x,t}(K\times[\varepsilon, T])$ as $k\to+\infty$, for each compact subset $K\subset M$ and for each $\varepsilon\in (0,T)$, and in $C_{loc}(M\times[0,T])$, to some function $u\in C^{2,1}_{x,t}(M\times(0,T])\cap C(M\times[0,T])$, which is a classical solution to problem \eqref{problema}. Moreover, from \eqref{eq512} we get
$$
0\le u\le\bar{u} \quad\text{in}\,\,\, M\times(0,+\infty).
$$
Hence the thesis follows.
 \end{proof}

\normalcolor

\section{On blow-up of solutions for large data}\label{fr}
In this section we discuss a blow-up result that can be obtained by standard tools. More precisely, we show that the solution to problem \eqref{problema} blows up, provided that $u_0$ is large enough, and
$f:[0, +\infty)\to [0, +\infty)$ is a locally Lipschitz, increasing, convex function fulfilling
\begin{equation*}
\int^{+\infty}\frac 1{f(s)} ds<+\infty\,\,.
\end{equation*}
We need to introduce some preliminary material. Let $o\in M$ be a reference point and $r(x)$ be the geodesic distance between $x$ and $o$.
For any $x\in M\setminus\{o\}$, denote by $\Rico$ the \textit{Ricci curvature} at $x$ in the radial direction $\frac{\partial}{\partial r}$. We assume that
\[\Rico(x) \geq - (N-1)\frac{\psi''(r(x))}{\psi(r(x))}\quad \text{ for all }\, x\in M\setminus\{o\}\,,\]
for some $\psi\in C^\infty((0, +\infty))\cap C^1([0, +\infty))$ such that $\psi'(0)=1, \psi(0)=0, \psi>0 \text{ in  } (0, +\infty)$
and
\begin{equation}\label{bound}\int^{+\infty}\frac{\int_0^r \psi^{N-1}(\xi)d\xi}{\psi^{N-1}(r)}dr=+\infty\,.\end{equation}
In view of such hypothesis, for problem \eqref{problema} comparison principle for bounded sub-- and supersolutions holds (see, e.g., \cite{Grig3}, \cite{Pu3}). Condition \eqref{bound} may be stated informally in a quite simpler way: a sufficient condition for this to hold is that
\begin{equation}\label{bound2}
\Rico(x) \geq -c r(x)^2 \quad \text{ for all }\, x\in M\setminus\{B_1(o)\}
\end{equation}
and a suitable $c>0$, as can be seen by choosing $\psi$ to be $e^{kr^2}$ in a neighborhood of infinity, for a suitable $k>0$.

Let $D$ be an open precompact subset of $M$ with smooth boundary.
By Kaplan's method (see \cite{K}) it can be proved that, for some $v_0\in C(\bar D), v_0\geq 0$ large enough, any solution $v$ to problem
\begin{equation}\label{N10}
\begin{cases}
\, v_t= \Delta v +\,f(v) & \text{in}\,\, D\times (0,T) \\
\, v = 0 & \text{in}\,\, \partial D\times (0, T)\\
\,\; v =v_0 &\text{in}\,\, D\times \{0\}\,
\end{cases}
\end{equation}
blows up in finite time. Now, consider $u_0\in C(M), u_0\geq0 $ with compact support. Take any $D$ as above containing the support of $u_0$ and set $v_0:=u_{0\lfloor{D}}$.
By choosing $u_0$ big enough, and so $v_0$, the solution $v$ to \eqref{N10}, corresponding to such $v_0$, blows up in a finite time, say $\tau>0$. Let
\[\underline u:=\begin{cases}
v& \text{in } D\times (0, \tau)\\
0& \text{in }(M\setminus D)\times (0, \tau)\,.
\end{cases}\]
By the maximum principle,
\[v\geq 0 \quad \text{ in }\,\, D\times (0, \tau)\,.\]
Hence,
\[\frac{\partial v}{\partial n}\leq 0 \quad \text{ in }\,\, \partial D\times (0, \tau),\]
$n$ being the outer unit normal vector to $\partial D$. This easily implies that, for any $0<T<\tau$, $\underline u$ is a bounded
weak subsolution to problem \eqref{problema}. So, by comparison principle, for any solution $u$ to problem \eqref{problema},
\[u\geq \underline u\quad \text{ in }\,\, M\times (0, T)\,.\]
Since $v$ blows in finite time, the same holds for $\bar u$ and so for $u$.

\par\bigskip\noindent
\textbf{Acknowledgments.}
The authors are members of the Gruppo Nazionale per l'Analisi Matematica, la Probabilit\`a e le loro Applicazioni (GNAMPA, Italy) of the Istituto Nazionale di Alta Matematica (INdAM, Italy) and are partially supported by the PRIN project 201758MTR2: ``Direct and Inverse Problems for Partial Differential Equations: Theoretical Aspects and Applications'' (Italy).

%
%

%


\end{document}